\documentclass[11pt,leqno]{amsart}
\usepackage{amssymb,amsmath,amsthm,latexsym}
\usepackage{graphicx}
\usepackage{epsfig}
\newtheorem{theorem}{Theorem}[section]
\newtheorem{lemma}[theorem]{Lemma}

\theoremstyle{remark}
\newtheorem{definition}[theorem]{Definition}
\newtheorem{remark}[theorem]{Remark}
\newtheorem{example}[theorem]{Example}

\newcommand\eL{{\mathcal L}}
\newcommand\M{{\mathcal M}}
\newcommand\R{{\mathcal R}}

\newcommand\Z{{\mathbb Z}}
\newcommand\xt{{(x_0,x_1,\ldots,x_n,x_{n+1})}}

\begin{document}

\title{Ternary quasigroups in knot theory}
\date{August 3, 2017}
\author{Maciej Niebrzydowski}
\address[Maciej Niebrzydowski]{Institute of Mathematics\\ 
Faculty of Mathematics, Physics and Informatics\\
University of Gda{\'n}sk, 80-308 Gda{\'n}sk, Poland}
\email{mniebrz@gmail.com}


\thispagestyle{empty}

\begin{abstract}
We show that some ternary quasigroups appear naturally as invariants of classical links and links on surfaces. We also note how to obtain from them invariants of Yoshikawa moves. In \cite{Ni17}, we defined homology theory for algebras satisfying two axioms derived from the third Reidemeister move.
In this paper, we show a degenerate subcomplex suitable for ternary quasigroups satisfying these axioms, and corresponding to the first Reidemeister move.
For such ternary quasigroups with an additional condition that the primary operation equals to the second division operation, we also define another  subcomplex, corresponding to the second flat Reidemeister move. Based on the normalized homology, we define cocycle invariants.
\end{abstract}

\maketitle

\section{Ternary quasigroups and knots}

\begin{definition}\label{quasigroup0}
A {\it (combinatorial) ternary quasigroup} is a set $X$ equipped with a ternary operation $T\colon X^3\to X$ such that for a quadruple $(x_1,x_2,x_3,x_0)$
of elements of $X$ satisfying $x_1x_2x_3T=x_0$, specification of any three elements of the quadruple determines the remaining one uniquely.
\end{definition}

\begin{definition}\label{quasigroup}
An {\it (equational) ternary quasigroup} $(X,T,\eL,\M,\R)$ is a set $X$ equipped with ternary operations $T$, $\eL$, $\M$, $\R$ satisfying the following pairs of equations for all $x_1$, $x_2$, $x_3\in X$:
\begin{align}
& (x_1x_2x_3T)x_2x_3\eL=x_1 & & (x_1x_2x_3\eL)x_2x_3T=x_1,\\
& x_1(x_1x_2x_3T)x_3\M=x_2  & & x_1(x_1x_2x_3\M)x_3T=x_2,\\
& x_1x_2(x_1x_2x_3T)\R=x_3 & & x_1x_2(x_1x_2x_3\R)T=x_3.
\end{align}
We call the operations $\eL$, $\M$ and $\R$ the left, middle, and right division, respectively.
\end{definition}

Each equational ternary quasigroup $(X,T,\eL,\M,\R)$ yields a combinatorial ternary quasigroup $(X,T)$, and each combinatorial ternary quasigroup $(X,T)$
yields an equational ternary quasigroup $(X,T,\eL,\M,\R)$ via
\[x_0x_2x_3\eL=x_1, x_1x_0x_3\M=x_2\ \textrm{and}\ x_1x_2x_0\R=x_3,\]
assuming $x_1x_2x_3T=x_0$. See \cite{BelSan,Bel,Sm08} for more details on $n$-ary quasigroups.

In \cite{Ni14} and \cite{Ni17}, we defined some ternary algebras giving classical link invariants via assignment of algebra elements to the regions of a link diagram on the plane, that is, to the components of the complement of a link projection. Here we consider a more general situation of link diagrams placed on compact oriented surfaces (as in Fig. \ref{surf}), with Reidemeister moves applied to them. If the diagrams have only flat crossings (with no over-under information attached to them), then we use flat Reidemeister moves, as in Fig. \ref{flatmoves}. Our considerations apply to classical links in $\mathbb{R}^3$ (or $\mathbb{S}^3$) as well. The subject of link diagrams on oriented compact surfaces is closely related to abstract link diagrams defined in \cite{KK00}, where the relation to virtual links (see \cite{Kau99}) was explained. Our ternary operations generalize checkerboard colorings of abstract link diagrams used in \cite{KaN02}. Diagrams on oriented surfaces were also studied in, for example, \cite{Naoko96} and \cite{CKS01}.

\begin{figure}
\begin{center}
\includegraphics[width=9 cm]{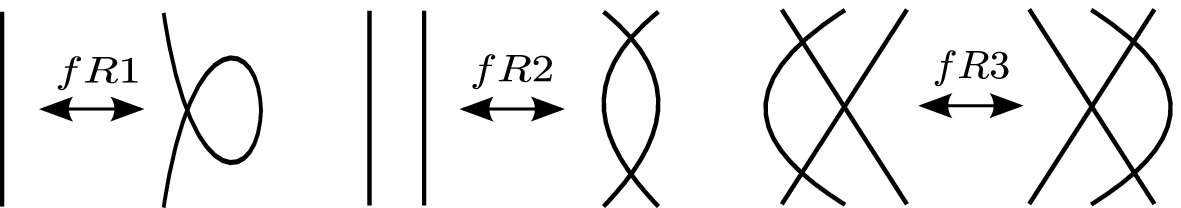}
\caption{}\label{flatmoves}
\end{center}
\end{figure}

\begin{figure}
\begin{center}
\includegraphics[height=5 cm]{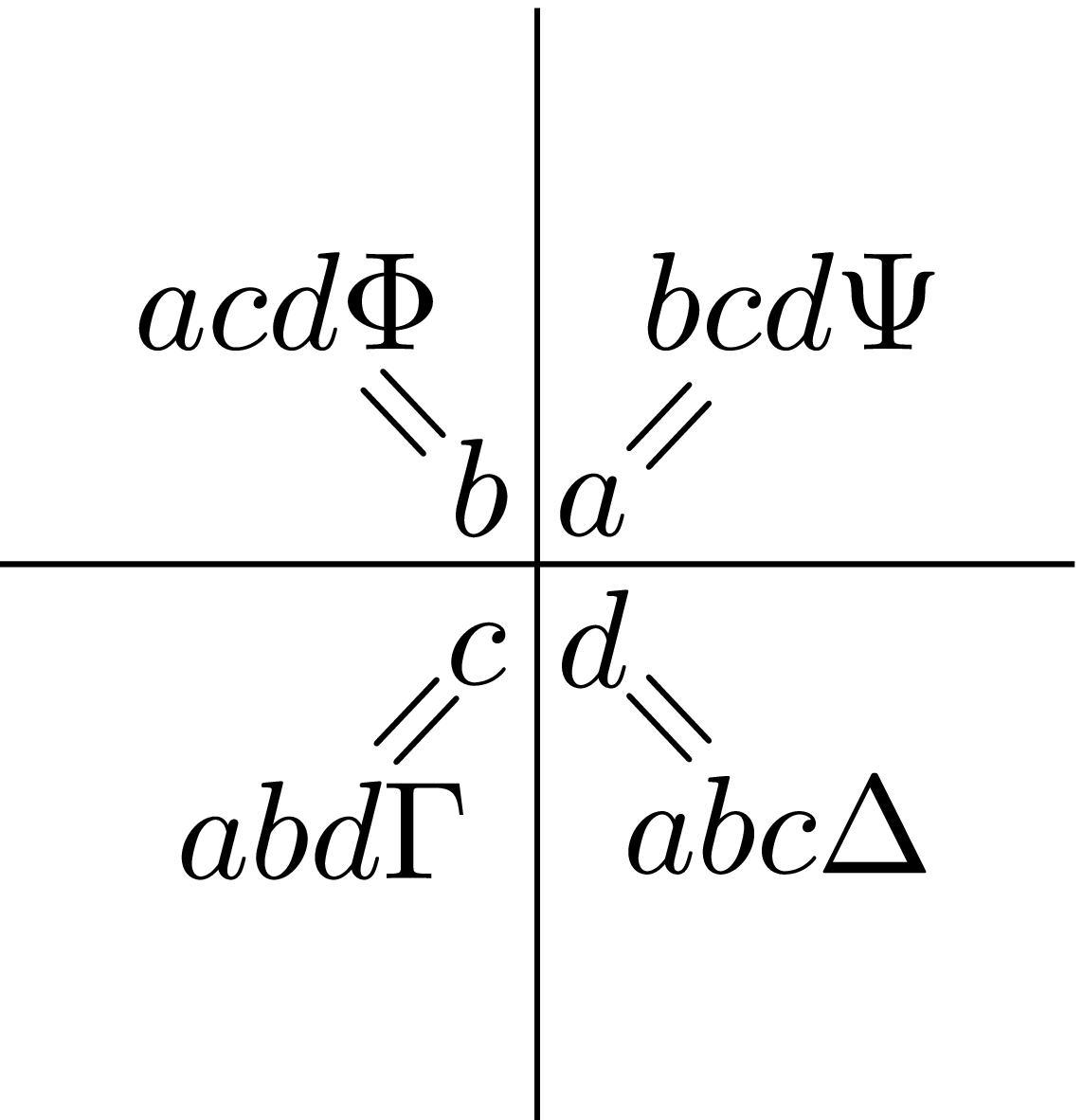}
\caption{}\label{coloringreg}
\end{center}
\end{figure}

\begin{definition}
Let $D$ be a link diagram (or a flat link diagram) on a compact oriented surface $F$, and let $R$ denote the connected components of $F\,\setminus$ universe of $D$. We call elements of $R$ {\it regions}. 
Let the four regions meeting at a crossing be assigned elements of some algebra (that is, a set with operations on it), such that the element in each corner can be expressed using the other three corner elements and some ternary operation. This situation is illustrated in Fig. \ref{coloringreg}. If there is a consistent way of assigning algebra elements to all regions of $D$ on $F$, then we call it a {\it ternary coloring}. 
\end{definition}

The above definition is general, and we only assume that we color with an algebra that has a family of ternary operations used in a predetermined way (they don't even need to be the same for every crossing). For now we assume that the colors of regions passing over or under a surface bridge are unchanged, as in Fig. \ref{ignoringbridge}. In this paper we will investigate the conditions that can be imposed on ternary colorings, so that their number, and homology classes assigned to them, are invariant under Reidemeister moves.

\begin{figure}
\begin{center}
\includegraphics[height=4 cm]{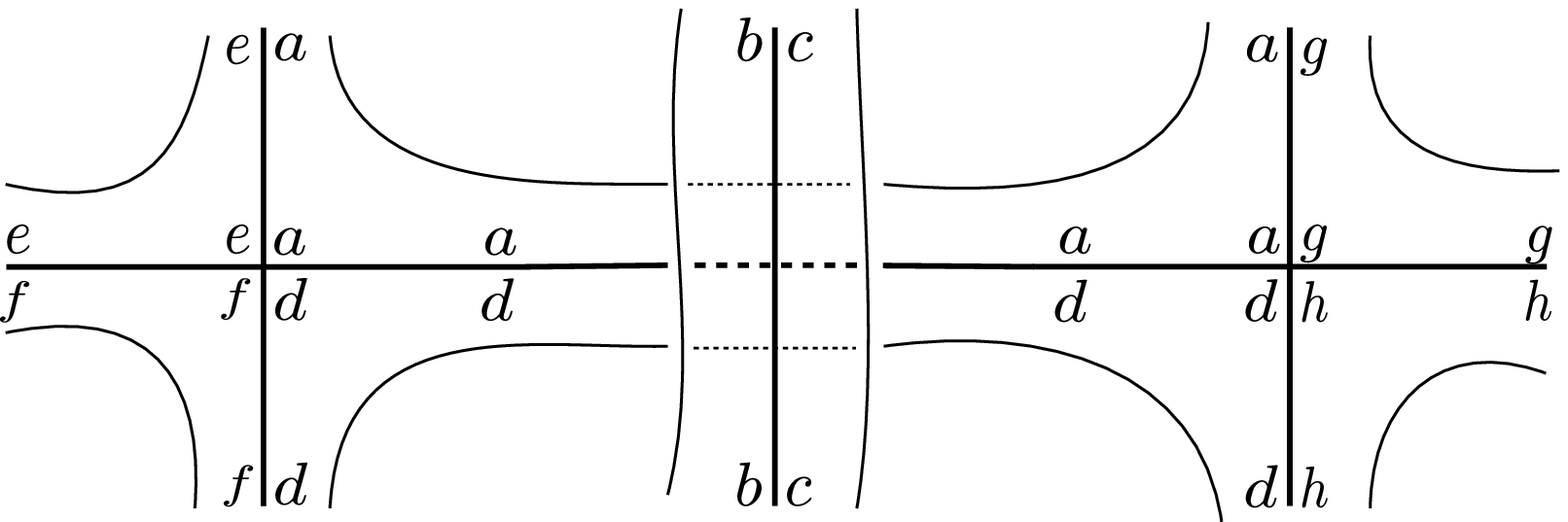}
\caption{}\label{ignoringbridge}
\end{center}
\end{figure}

\begin{figure}
\begin{center}
\includegraphics[height=4 cm]{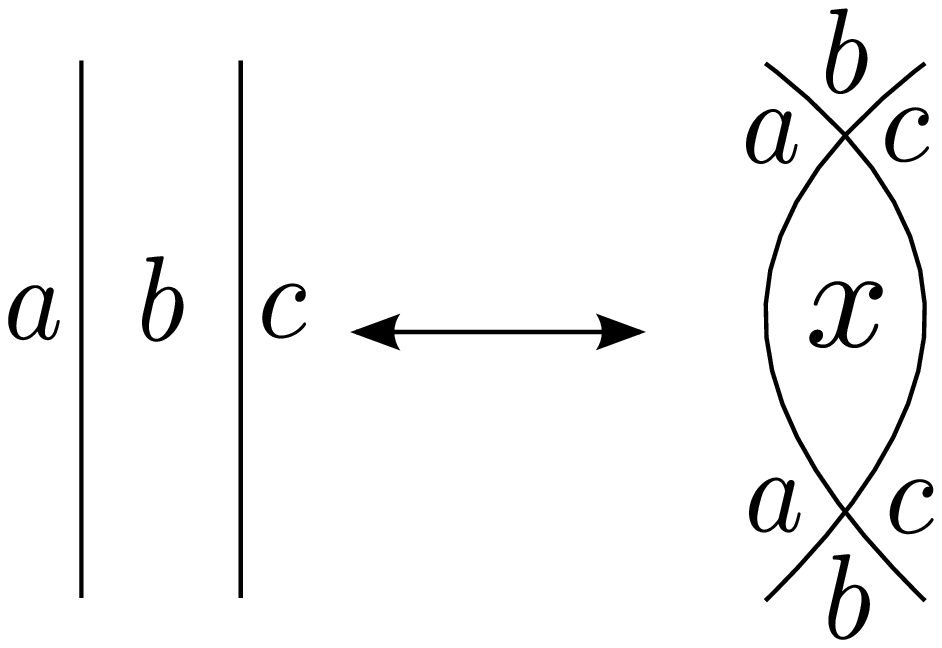}
\caption{}\label{R2gen}
\end{center}
\end{figure}

\begin{figure}
\begin{center}
\includegraphics[height=3.5 cm]{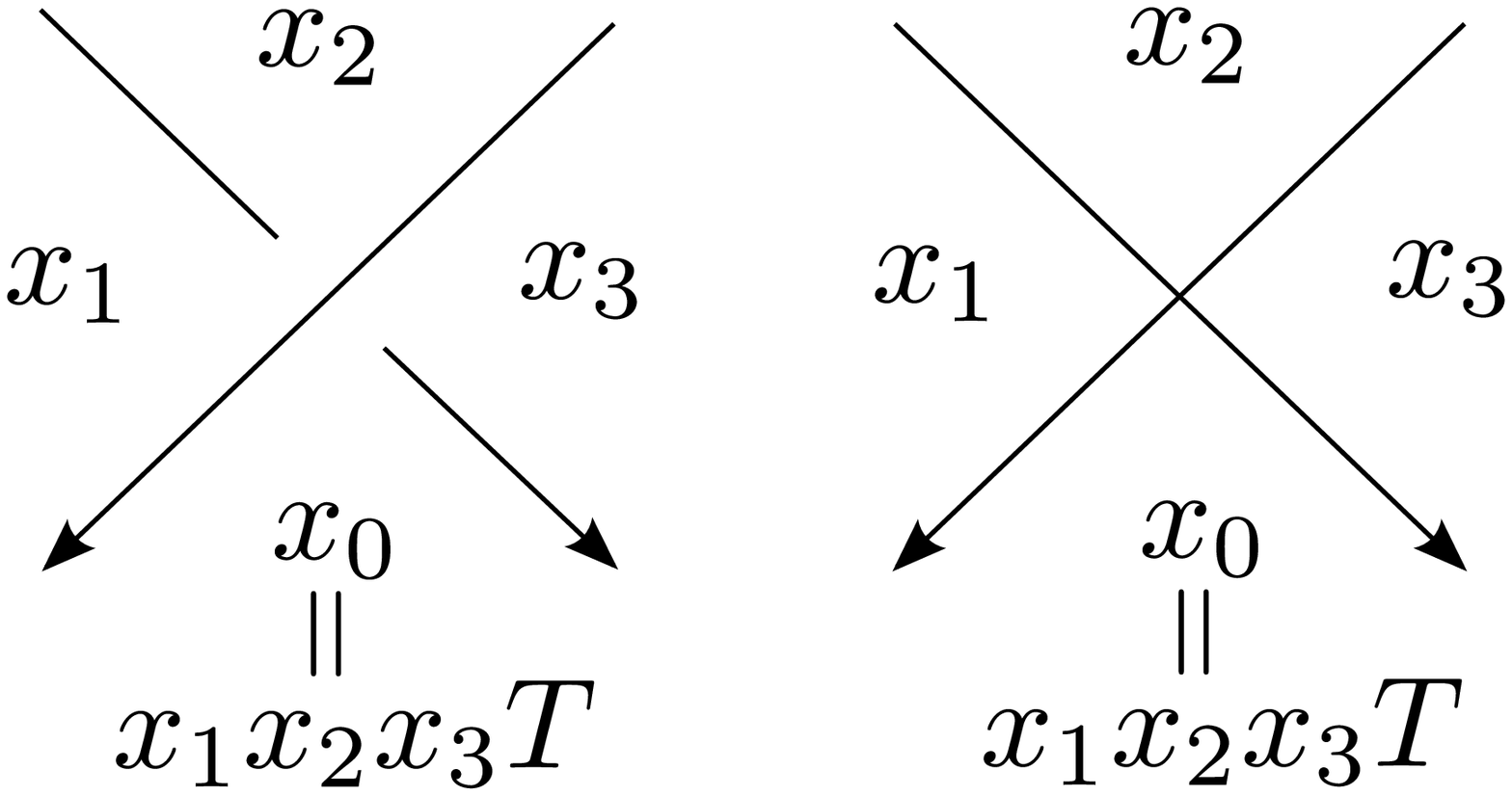}
\caption{}\label{combquasi}
\end{center}
\end{figure}

In an oriented link diagram, each of the four corners around a crossing has its own characterization, i.e., there is a corner adjacent to the two incoming edges, a corner $r$ touching the two outgoing edges, a corner having $r$ on the right, and the one that has it on the left. Fig. \ref{R2gen} represents a schematic colored second Reidemeister move, without specifying the types of crossings. Depending on the orientation, the corner colored by $x$ could be of any of the four types. If we are to have a coloring, then $x$ must exist for any $a$, $b$ and $c$. If the number of colorings of a diagram is to be unchanged by the move, then $x$ has to be unique. Thus, we reach a combinatorial definition of a ternary quasigroup, and we use its primary operation $T$ to color the corner adjacent to the outgoing edges of a flat or a positive classical crossing, see Fig. \ref{combquasi}. The bottom crossing of Fig. \ref{R2gen} yields additional conditions depending on its type.  

\begin{figure}
\begin{center}
\includegraphics[height=9 cm]{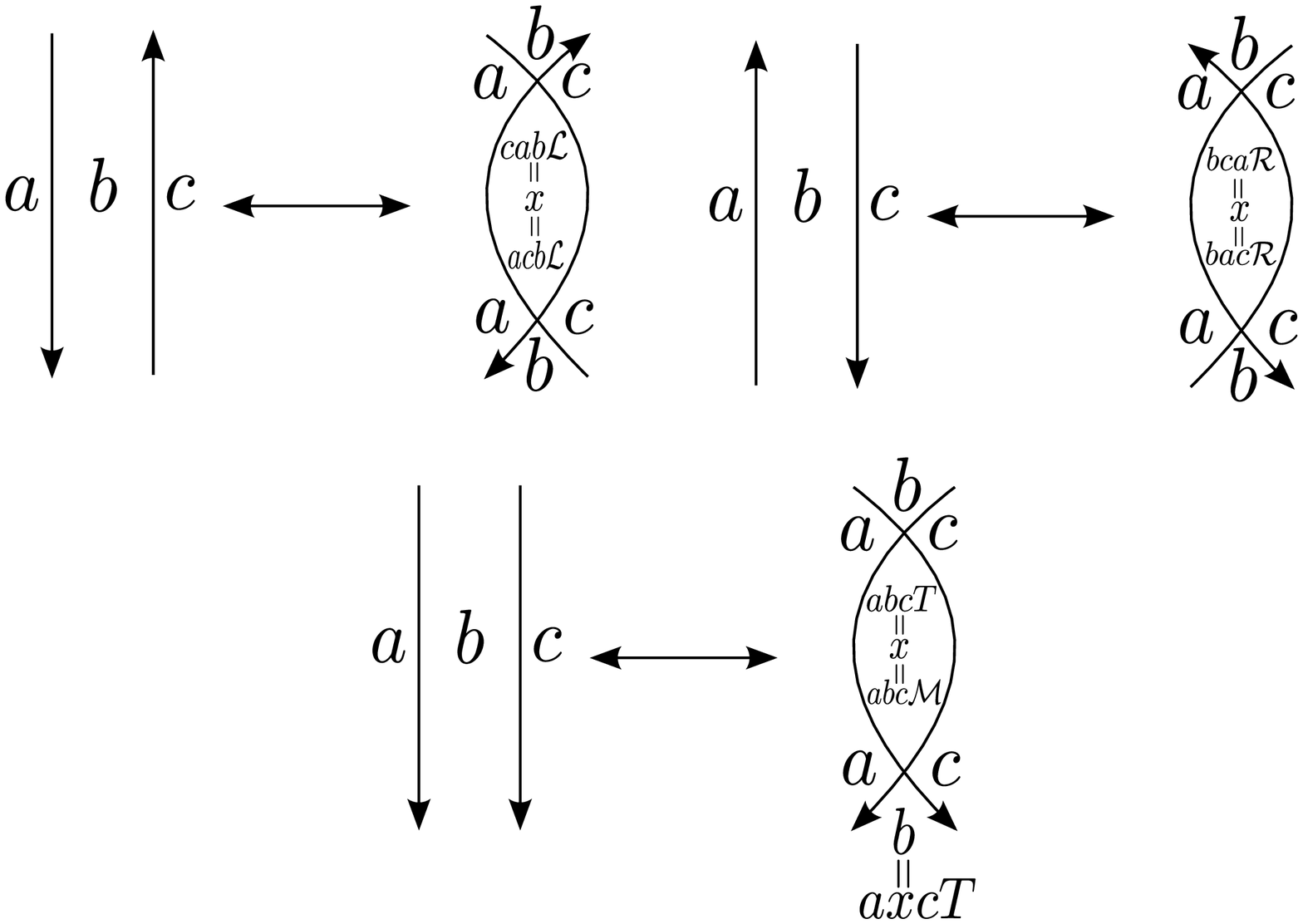}
\caption{}\label{fR2}
\end{center}
\end{figure} 

Let us use the convention from Fig. \ref{combquasi} to color different versions of the second flat Reidemeister move, see Fig. \ref{fR2}. We observe that in the move with parallel orientation of strands
\[
x=abcT=abc\M.
\]
Thus, for flat links diagrams, we will be using ternary quasigroups with $T=\M$.
Also,
\begin{align*}
& cab\eL=acb\eL,\\
& bca\R=bac\R,
\end{align*}
and the permutation
\[
a\square cT\colon X\to X,\ x\mapsto axcT
\] is an involution, for arbitrary elements
$a$, $b$, $c$. These last three observations are really just consequences of the equality $T=\M$.

\begin{figure}
\begin{center}
\includegraphics[height=6.5 cm]{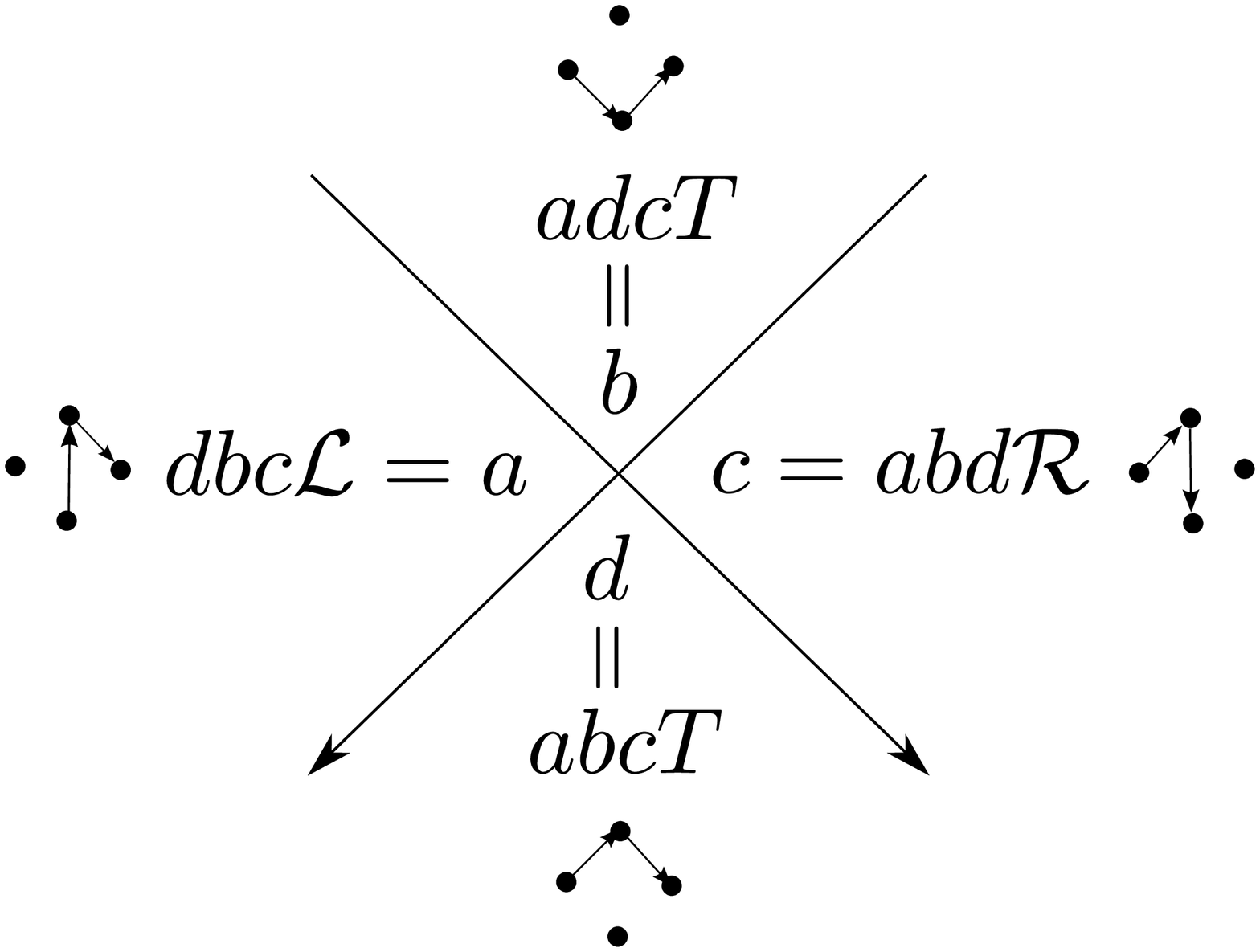}
\caption{}\label{fquasi}
\end{center}
\end{figure} 

To summarize the above considerations, given a flat link diagram on a compact oriented surface $F$, its ternary coloring with colors from a ternary quasigroup $(X,T,\eL,T,\R)$ is determined as in Fig. \ref{fquasi}. The inputs for the operations are taken in the order indicated in the Figure \ref{fquasi} by the dot diagrams. 

From substitutions in the expressions around the crossing in Fig. \ref{fquasi},  we can recover the defining equations in a ternary quasigroup (with $T=\M$).
\begin{align*}
& a=dbc\eL=(abcT)bc\eL=d(adcT)c\eL=(adcT)dc\eL=db(abd\R)\eL,\\
& b=adcT=(dbc\eL)dcT=(bdc\eL)dcT=a(abcT)cT=ad(abd\R)T=ad(adb\R)T,\\
& c=abd\R=(dbc\eL)bd\R=a(adcT)d\R=ad(adcT)\R=ab(abcT)\R,\\
& d=abcT=(dbc\eL)bcT=a(adcT)cT=ab(abd\R)T.
\end{align*}
The equations above that connect
$\eL$ and $\R$ follow from the usual relations between divisions in a ternary quasigroup, together with the aforementioned two-coordinate commutativities.

The invariance under the first Reidemeister move follows from the combinatorial definition of ternary quasigroup: the new color needed for the coloring of the kink exists and is unique (see the right side of Fig. \ref{degR1}), so the number of colorings stays the same.

\begin{figure}
\begin{center}
\includegraphics[height=6.5 cm]{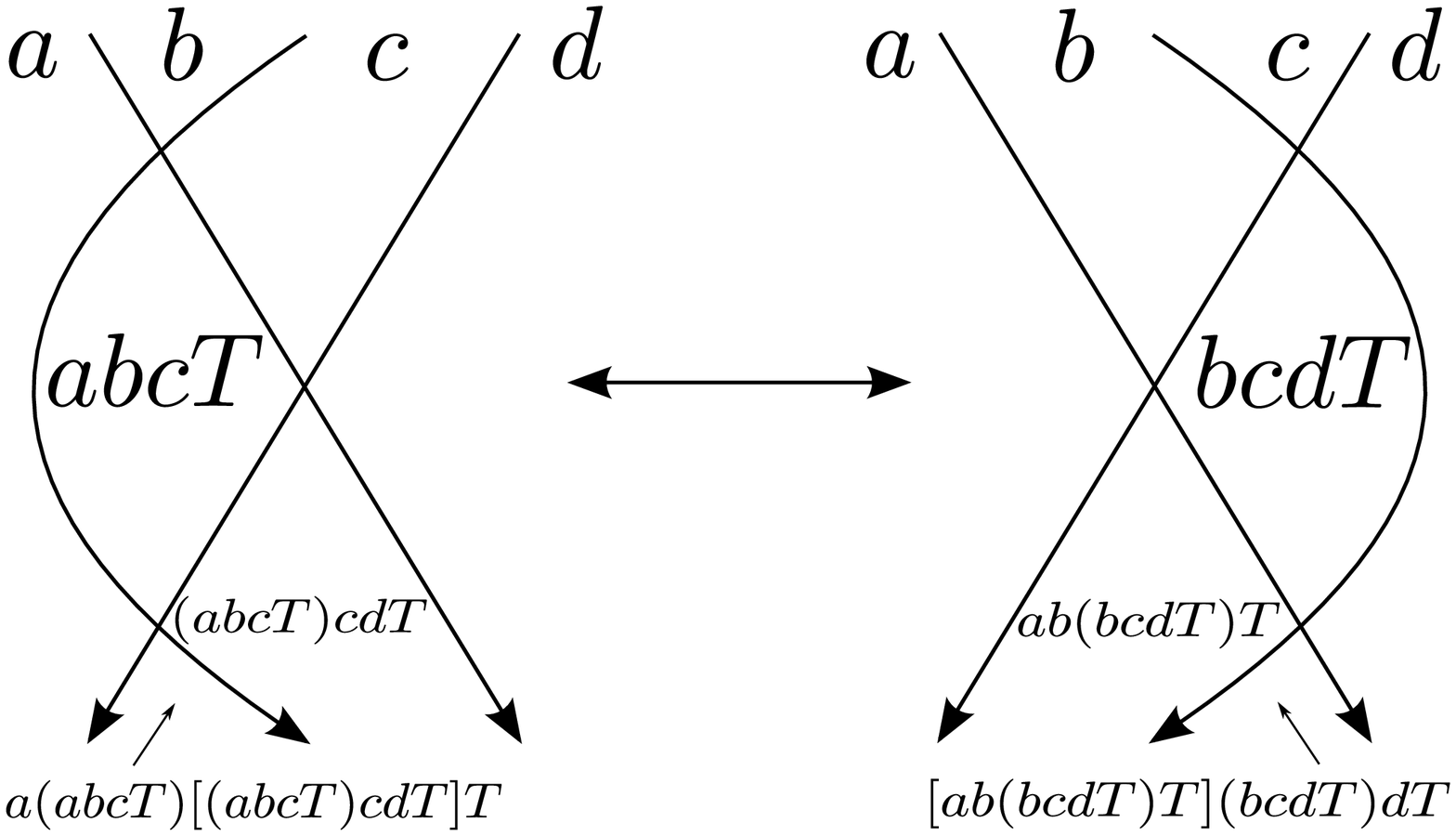}
\caption{}\label{fR3}
\end{center}
\end{figure}

The third flat Reidemeister move, illustrated in Fig. \ref{fR3}, yields two axioms: 
\begin{equation}
\forall_{a,b,c,d\in X} \quad (abcT)cdT=[ab(bcdT)T](bcdT)dT \tag{A3L} 
\end{equation}
\begin{equation}
\forall_{a,b,c,d\in X} \quad ab(bcdT)T=a(abcT)[(abcT)cdT]T  \tag{A3R}
\end{equation}

Note that the right side of A3L (resp. A3R) is obtained from the left side of A3L (resp. A3R) by substitution $c\mapsto bcdT$ (resp. $b\mapsto abcT$ ). 

Thus, to obtain invariants of flat link diagrams under flat Reidemeister moves, one can use ternary quasigroups $(X,T,\eL,T,\R)$ satisfying A3L and A3R. In this paper, we will call such a quasigroup {\it involutory knot-theoretic ternary quasigroup}, and denote it simply by IKTQ.

\begin{figure}
\begin{center}
\includegraphics[height=6.5 cm]{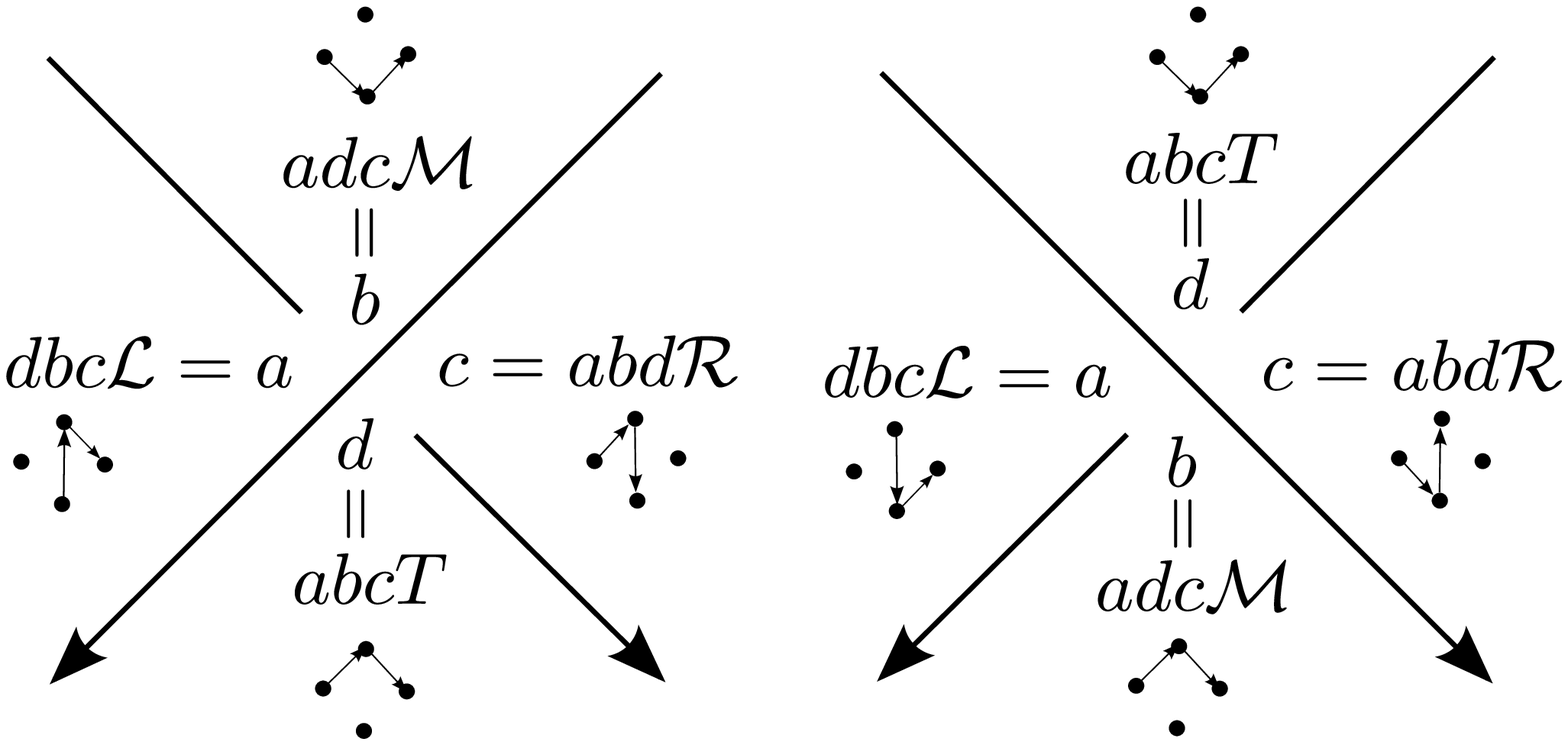}
\caption{}\label{quasi}
\end{center}
\end{figure}

Now we turn our attention to link diagrams on compact, oriented surfaces. Assuming the convention from Fig. \ref{combquasi} for a positive crossing, we need to consider what should be done with a negative crossing. If we try to use for it another set of quasigroup operations (i.e., another primary operation and its three divisions), we need to consider different versions of the second Reidemeister move, in which the region between the two crossings provides connections between the operations associated with a positive crossing and the ones for a negative crossing. There comes a quick realization that the new operations can be expressed using the ones corresponding to a positive crossing, as in Fig. \ref{quasi}. In a negative crossing, the primary operation $T$ is used for the corner adjacent to the incoming edges. Again, we can recover the equational definition of ternary quasigroup $(X,T,\eL,\M,\R)$ by performing substitutions in the equations around the crossings in Fig. \ref{quasi}:
\begin{align*}
& a=dbc\eL=(abcT)bc\eL=d(adc\M)c\eL=db(abd\R)\eL,\\
& b=adc\M=(dbc\eL)dc\M=a(abcT)c\M=ad(abd\R)\M,\\
& c=abd\R=(dbc\eL)bd\R=a(adc\M)d\R=ab(abcT)\R,\\
& d=abcT=(dbc\eL)bcT=a(adc\M)cT=ab(abd\R)T.
\end{align*}
The equations above that do not include $T$ are the usual relations between divisions in a ternary quasigroup. 

The invariance of the number of ternary colorings under the
first and second Reidemeister moves does not require any additional axioms, and follows from the quasigroup structure. The invariance under the third Reidemeister move (in which all the crossings are positive) requires the addition of axioms A3L and A3R. We will call a ternary quasigroup $(X,T,\eL,\M,\R)$ satisfying A3L and A3R a knot-theoretic ternary quasigroup, and denote it by KTQ.

\begin{figure}
\begin{center}
\includegraphics[height=2.5 cm]{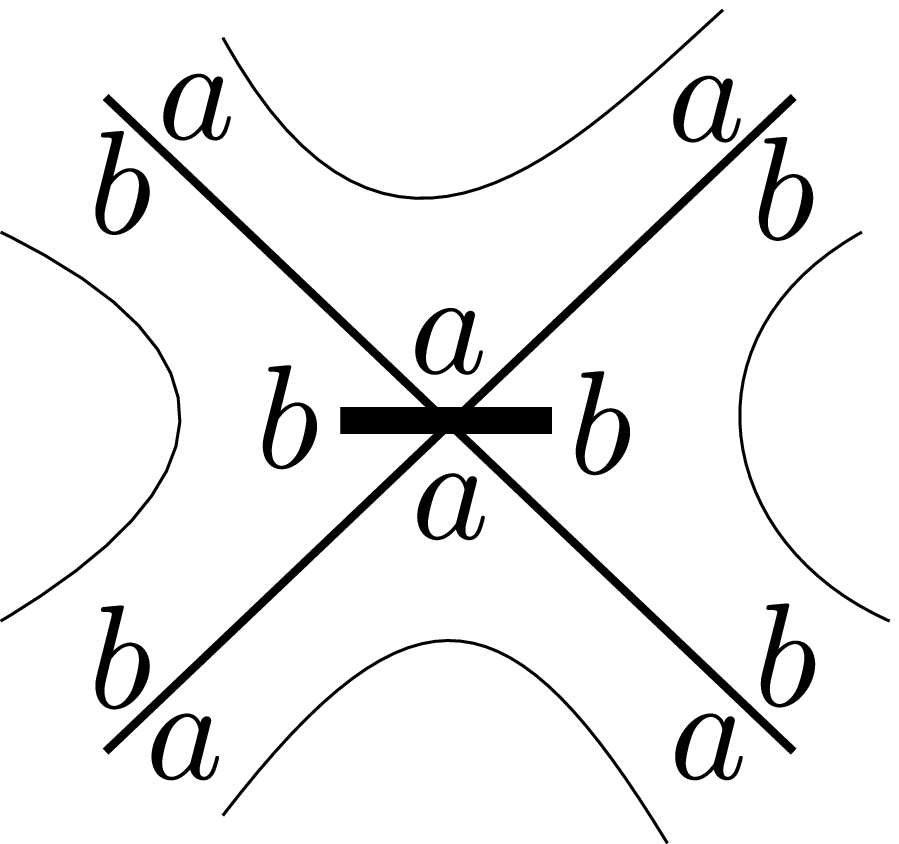}
\caption{}\label{yoshikawa}
\end{center}
\end{figure}

\begin{remark}
We note that one can use KTQ colorings for oriented version of Yoshikawa diagrams on a compact oriented surface $F$, or simply on a plane. See \cite{Yo94} for a description of classical Yoshikawa diagrams, and \cite{KJL15} for a generating set of moves on them. If the convention from the Fig. \ref{yoshikawa} is used around markers (that is, opposite corners are assigned the same color), then the number of ternary colorings is not changed by the Yoshikawa moves on $F$. A special case of such colorings for classical Yoshikawa diagrams was investigated in \cite{KN17}.
\end{remark}

\begin{figure}
\begin{center}
\includegraphics[height=6 cm]{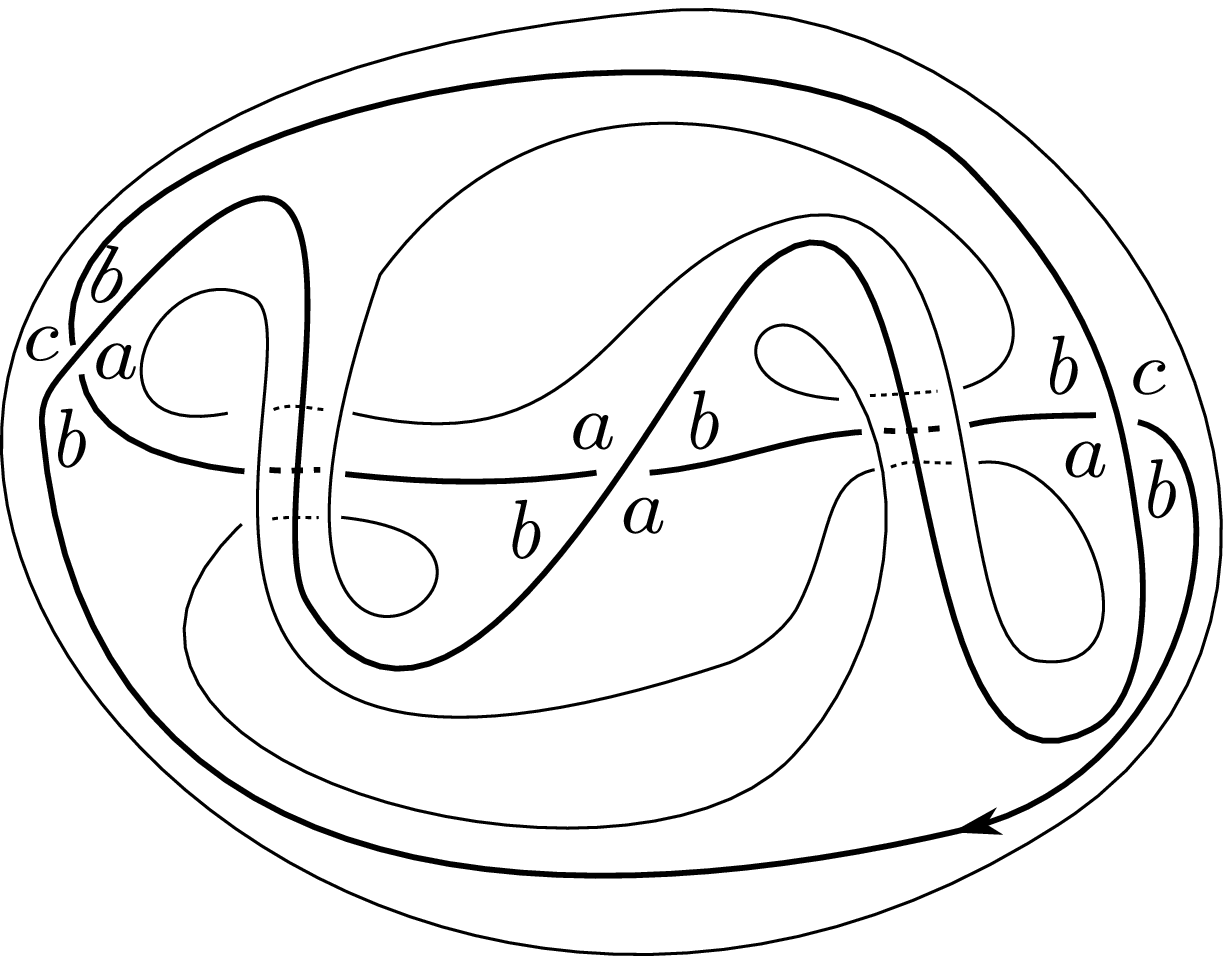}
\caption{}\label{surf}
\end{center}
\end{figure}

\begin{remark}
Consider the category of knot-theoretic ternary quasigroups. Let $D$ be a link diagram on a compact oriented surface $F$. To every such $D$ one can associate a  KTQ($D$) in the following way: its generators correspond to the regions $R$ of $D$ on $F$ and relations correspond to crossings and are as in Fig. \ref{quasi} (one relation among four can be chosen, as they are equivalent). Then two diagrams on $F$ that differ by Reidemeister moves on $F$ have isomorphic KTQ's.
Analogous constructions yield IKTQ($D$) for a flat link diagram $D$ on $F$, and KTQ($D$) for a Yoshikawa diagram $D$ on $F$. Then, any KTQ (resp. IKTQ)  coloring of a diagram can be viewed
as a homomorphism from KTQ($D$) (resp. IKTQ($D$)) to the KTQ (resp. IKTQ) used for the coloring.
\end{remark}

\begin{example}
A KTQ assigned to a knot diagram in 
Fig. \ref{surf} has presentation
\begin{align*}
KTQ(D)& =\langle a,b,c \ |\ babT=c, abaT=b, bcbT=a \rangle\\
& =\langle a,b \ |\ abaT=b, b(babT)bT=a \rangle.
\end{align*}
\end{example}

\section{Homology of KTQs}
In \cite{Ni17} we introduced a homology theory for arbitrary algebras satisfying axioms A3L and A3R. In this paper, we define a subcomplex that works for ternary quasigroups satisfying these axioms (i.e., KTQs). Then we define the homology of KTQs as the normalized homology.
First we recall the construction from \cite{Ni17}. 

\begin{definition}
For a given ternary algebra $X$ satisfying the axioms A3L and A3R,
let $C_n(X)=\mathbb{Z}\langle X^{n+2}\rangle$ be the free abelian group generated by $(n+2)$-tuples $(x_0,x_1,\ldots, x_n,x_{n+1})$ of elements of $X$, for $n\geq -1$, and let $C_{-2}(X)=\mathbb{Z}$. We define a boundary homomorphism
$\partial_n\colon C_n(X) \to C_{n-1}(X)$, by 
\begin{align*}
& \partial_{-1}(x_0)=0,\\
& \partial_0(x_0,x_1)=x_1-x_0,
\end{align*}
and for $n>0$:
\[\partial_n=\partial_n^L-\partial_n^R,\]
where $\partial_n^L\colon C_n(X) \to C_{n-1}(X)$ and 
$\partial_n^R\colon C_n(X) \to C_{n-1}(X)$ have the following definitions. 
\begin{equation*}
\partial_n^L\xt=\sum_{i=0}^n(-1)^i d_i^{n,L}\xt,
\end{equation*}
where $d_i^{n,L}$ is defined inductively by
\begin{align*}
& d_0^{n,L}\xt=(x_1,\ldots,x_n,x_{n+1}),\\
& d_{i}^{n,L}\xt=d_{i-1}^{n,L}\xt[x_i\mapsto x_{i-1}x_ix_{i+1}T]\\
& =d_{i-1}^{n,L}(x_0,\ldots,x_{i-1},x_{i-1}x_ix_{i+1}T,x_{i+1},\ldots,x_{n+1}),
\end{align*}
for $i\in\{1,\ldots,n\}$.
\begin{equation*}
\partial_n^R\xt=\sum_{i=0}^n(-1)^i d_i^{n,R}\xt,
\end{equation*}
where $d_i^{n,R}$ is defined inductively by
\begin{align*}
& d_0^{n,R}\xt=(z_0,z_1\ldots,z_n),\\
& \textrm{where}\ z_0=x_0, z_i=z_{i-1}x_ix_{i+1}T,\ \textrm{for}\ i=1,\ldots,n,\ \textrm{and}\\
& d_{i}^{n,R}\xt=d_{i-1}^{n,R}\xt[x_{i-1}x_ix_{i+1}T\mapsto x_i],\\
& \textrm{for}\ i\in\{1,\ldots,n\}.
\end{align*}
That is, the formula for $d_{i}^{n,R}\xt$ is obtained from \\
$d_{i-1}^{n,R}\xt$ by replacing $x_{i-1}x_ix_{i+1}T$ with $x_i$. 
\end{definition}

\begin{example}
In low dimensions the differential takes the following form:
\begin{align*}
\partial_1(a,b,c)& =(b,c)-(a,abcT)\\
& -(abcT,c)+(a,b),\\
\partial_2(a,b,c,d) & =(b,c,d)-(a,abcT,(abcT)cdT)\\
& -(abcT,c,d)+(a,b,bcdT)\\
& +(ab(bcdT)T,bcdT,d)-(a,b,c),\\
\partial_3(a,b,c,d,e)& =(b,c,d,e)-(a,abcT,(abcT)cdT,[(abcT)cdT]deT)\\
& -(abcT,c,d,e)+(a,b,bcdT,(bcdT)deT)\\
& +(ab(bcdT)T,bcdT,d,e)-(a,b,c,cdeT)\\
& -(ab[bc(cdeT)T]T,bc(cdeT)T,cdeT,e)+(a,b,c,d).
\end{align*}
\end{example}

Let $x$ denote $\xt$. There is another description of $d_i^{n,L}$ and $d_i^{n,R}$, in which their coordinates are defined inductively. 
\[
d_i^{n,L}=(d_{i,1}^{n,L},\ldots,d_{i,k}^{n,L},\ldots,d_{i,n+1}^{n,L})
\] 
is calculated from right to left. For $i\in\{0,\ldots,n\}$ and $k\in\{1,\ldots,n+1\}$,
\begin{equation*}\label{leftcoord}
d_{i,k}^{n,L}x = \left\{
\begin{array}{rl}
x_{k-1}x_k(d_{i,k+1}^{n,L}x)T & \text{if } k\leq  i\\
x_k & \text{if } k > i.
\end{array} \right.
\end{equation*}
\[
d_i^{n,R}=(d_{i,0}^{n,R},\ldots,d_{i,k}^{n,R},\ldots,d_{i,n}^{n,R})
\] 
is calculated from left to right. For $i\in\{0,\ldots,n\}$ and $k\in\{0,\ldots,n\}$,
\begin{equation*}\label{rightcoord}
d_{i,k}^{n,R}x = \left\{
\begin{array}{rl}
(d_{i,k-1}^{n,R}x)x_kx_{k+1}T & \text{if } k > i\\
x_k & \text{if } k \leq i.
\end{array} \right.
\end{equation*}

Now we are ready to define a subcomplex corresponding to KTQs.

\begin{definition}\label{degs}
For a ternary quasigroup $(X,T,\eL,\M,\R)$ satisfying axioms A3L and A3R, and for $n\geq 1$,
let $C_n^D(X)$ denote the free abelian group generated by $(n+2)$-tuples 
$x=(x_0,x_1,\ldots, x_n,x_{n+1})$ of elements of $X$ satisfying at least one of the conditions:
\begin{itemize}
\item[(D1)] $x$ contains $a$, $b$, $abb\R$ on three consecutive coordinates, for some $a$ and $b\in X$;
\item[(D2)] $x$ contains $bba\eL$, $b$, $a$ on three consecutive coordinates, for some $a$ and $b\in X$. 
\end{itemize}
For $n<1$, we take $C_n^D(X)=0$.
\end{definition}

\begin{theorem}\label{degL}
For a KTQ,
\[\partial_n^L(C_n^D)\subset C_{n-1}^D.\]
\end{theorem}
\begin{proof}
We note that $d_i^{n,L}x$ contains at the end the sequence
$x_{i+1},\ldots,x_{n+1}$. Suppose that $x$ is such that $x_j$ is the first element of the triple $a$, $b$, $abb\R$, or of the triple 
$bba\eL$, $b$, $a$. Then this triple occurs also in all
$d_i^{n,L}x$ with $i\in\{0,\ldots, j-1\}$.
Now consider $i=j+1$:
\begin{align*}
&d_{j+1}^{n,L}(x_0,\ldots,x_{j-1},a,b,abb\R,x_{j+3},\ldots,x_{n+1})\\
& =d_j^{n,L}(x_0,\ldots,x_{j-1},a,ab(abb\R)T,abb\R,x_{j+3},\ldots,x_{n+1})\\
& =d_j^{n,L}(x_0,\ldots,x_{j-1},a,b,abb\R,x_{j+3},\ldots,x_{n+1}),\\
&d_{j+1}^{n,L}(x_0,\ldots,x_{j-1},bba\eL,b,a,x_{j+3},\ldots,x_{n+1})\\
& =d_j^{n,L}(x_0,\ldots,x_{j-1},bba\eL,(bba\eL)baT,a,x_{j+3},\ldots,x_{n+1})\\
& =d_j^{n,L}(x_0,\ldots,x_{j-1},bba\eL,b,a,x_{j+3},\ldots,x_{n+1}).
\end{align*}
However, in $\partial_n^L$, $d_{j}^{n,L}x$ and $d_{j+1}^{n,L}x$ appear with opposite signs, so there is a reduction. Now let $j+2\leq i \leq n$. We will show that if $x$ satisfies the condition (D1), then 
$d_i^{n,L}x$ satisfies (D2); more precisely, it contains the triple
\[(d^{n,L}_{i,j+2}x)(d^{n,L}_{i,j+2}x)(d^{n,L}_{i,j+3}x)\eL,\ d^{n,L}_{i,j+2}x,\ 
d^{n,L}_{i,j+3}x.\]
First, note that in the (D1) case we have
\begin{align*}
d_{i,j+1}^{n,L}x & =x_jx_{j+1}(d_{i,j+2}^{n,L}x)T=
x_jx_{j+1}(x_{j+1}x_{j+2}(d_{i,j+3}^{n,L}x)T)T\\
& =ab[b(abb\R)(d_{i,j+3}^{n,L}x)T]T,
\end{align*}
and 
\[
d_{i,j+2}^{n,L}x =x_{j+1}x_{j+2}(d_{i,j+3}^{n,L}x)T
=b(abb\R)(d_{i,j+3}^{n,L}x)T.
\]
Now we will show the equation
\[
d_{i,j+1}^{n,L}x=(d^{n,L}_{i,j+2}x)(d^{n,L}_{i,j+2}x)(d^{n,L}_{i,j+3}x)\eL,
\]
by transforming the equality
\[
b(abb\R)(d_{i,j+3}^{n,L}x)T=b(abb\R)(d_{i,j+3}^{n,L}x)T.
\]
First, from the definition of ternary quasigroup
\[
[ab(abb\R)T](abb\R)(d_{i,j+3}^{n,L}x)T=b(abb\R)(d_{i,j+3}^{n,L}x)T.
\]
Next, using the axiom A3L
\[
\{ab[b(abb\R)(d_{i,j+3}^{n,L}x)T]T\}[b(abb\R)(d_{i,j+3}^{n,L}x)T](d_{i,j+3}^{n,L}x)T=b(abb\R)(d_{i,j+3}^{n,L}x)T.
\]
Finally, using the defining equations of ternary quasigroups
\[
[b(abb\R)(d_{i,j+3}^{n,L}x)T][b(abb\R)(d_{i,j+3}^{n,L}x)T](d_{i,j+3}^{n,L}x)\eL=ab[b(abb\R)(d_{i,j+3}^{n,L}x)T]T,
\]
which is the equation that we wanted to show. Now suppose that $x$ satisfies (D2), in which case we have
\[
d_{i,j+1}^{n,L}x=(bba\eL)b[ba(d_{i,j+3}^{n,L}x)T]T,
\]
and 
\[
d_{i,j+2}^{n,L}x = ba(d_{i,j+3}^{n,L}x)T.
\]
Again, using the quasigroup properties and the axiom A3L, we will show the equation
\[
d_{i,j+1}^{n,L}x=(d^{n,L}_{i,j+2}x)(d^{n,L}_{i,j+2}x)(d^{n,L}_{i,j+3}x)\eL,
\]
by modifying the equality
\[
ba(d_{i,j+3}^{n,L}x)T=ba(d_{i,j+3}^{n,L}x)T
\]
as follows:
\[
[(bba\eL)baT]a(d_{i,j+3}^{n,L}x)T=ba(d_{i,j+3}^{n,L}x)T,
\]
\[
\{(bba\eL)b[ba(d_{i,j+3}^{n,L}x)T]T\}[ba(d_{i,j+3}^{n,L}x)T](d_{i,j+3}^{n,L}x)T=ba(d_{i,j+3}^{n,L}x)T,
\]
\[
[ba(d_{i,j+3}^{n,L}x)T][ba(d_{i,j+3}^{n,L}x)T](d_{i,j+3}^{n,L}x)\eL=
(bba\eL)b[ba(d_{i,j+3}^{n,L}x)T]T.
\]
The last line is the desired equation.
\end{proof}

\begin{definition}
Given a ternary operation $T$, let $\hat{T}$ be
defined by \[xyz\hat{T}=zyxT.\]
\end{definition}

The following two lemmas are obtained by checking the definitions.

\begin{lemma}\label{revquasi}
$(X,T,T_1,T_2,T_3)$ is a ternary quasigroup $\iff$ $(X,\hat{T},\hat{T_3},\hat{T_2},\hat{T_1})$ is a ternary quasigroup.
\end{lemma}

\begin{lemma}\label{rev1}
$T$ satisfies A3R if and only if $\hat{T}$ satisfies A3L.\\
$T$ satisfies A3L if and only if $\hat{T}$ satisfies A3R.
\end{lemma}

Let $y^r$ denote reversing the order of the elements of the tuple $y$;
linear extension of this transformation will be denoted with the same symbol.
When two or more operators are considered, their symbols are added to the differentials, as in the following lemma from \cite{Ni17}.
\begin{lemma}\label{convert}
$d_i^{n,R,T}x=(d_{n-i}^{n,L,\hat{T}}x^r)^r$ for $i\in\{0,\ldots,n\}$.
\end{lemma}

\begin{theorem}\label{degR}
For a KTQ,
\[\partial_n^R(C_n^D)\subset C_{n-1}^D.\]
\end{theorem}
\begin{proof}
Note that from Lemma \ref{revquasi} we have
\begin{align*}
& (\ldots,bbaT_1,b,a,\ldots)^r=(\ldots,bba\eL,b,a,\ldots)^r=(\ldots,a,b,bba\eL,\ldots)\\
& =(\ldots,a,b,abb\hat{\eL},\ldots)=(\ldots,a,b,abb\hat{T}_3,\ldots),
\end{align*}
and
\begin{align*}
& (\ldots,a,b,abbT_3,\ldots)^r=(\ldots,a,b,abb\R,\ldots)^r=(\ldots,abb\R,b,a,\ldots)\\
& =(\ldots,bba\hat{\R},b,a,\ldots)=(\ldots,bba\hat{T}_1,b,a,\ldots).
\end{align*}
In other words, degeneracies of type (D1) (resp. (D2)) for the operation $T$ are transformed into degeneracies of type (D2) (resp. (D1)) for the operation $\hat{T}$, and vice versa, by the operator $r$.

From Lemma \ref{convert} it follows that
\[ \partial_n^{R,T}x=(-1)^n(\partial_n^{L,\hat{T}}x^r)^r.\]
Theorem \ref{degR} now follows from Lemma \ref{rev1} and Theorem \ref{degL}.
\end{proof}

From the Theorems \ref{degL} and \ref{degR} follows

\begin{theorem}\label{degLR}
For a KTQ, 
\[\partial_n(C_n^D)\subset C_{n-1}^D.\]
\end{theorem}

\begin{definition}
We proved that, for $X$ being a KTQ, $(C_n^D(X),\partial_n^L)$, $(C_n^D(X),\partial_n^R)$ and 
$(C_n^D(X),\partial_n)$ are chain subcomplexes of 
$(C_n(X),\partial_n^L)$, $(C_n(X),\partial_n^R)$, and 
$(C_n(X),\partial_n)$, respectively. We call their homology {\it left degenerate}, {\it right degenerate}, and {\it degenerate}, and denote it by $H^{LD}(X)$, $H^{RD}(X)$, and $H^D(X)$, respectively. We define quotient complexes
\begin{align*} 
& (C_n^{N}(X),\partial_n^L)=(C_n(X)/C_n^D(X),\partial_n^L),\\
& (C_n^{N}(X),\partial_n^R)=(C_n(X)/C_n^D(X),\partial_n^R),\\ 
& (C_n^{N}(X),\partial_n)=(C_n(X)/C_n^D(X),\partial_n),
\end{align*}
with induced differentials (and the same notation). We call the homology of these complexes {\it left normalized},
{\it right normalized}, and {\it normalized},
and denote it by $H^{LN}(X)$, $H^{RN}(X)$, and $H^N(X)$, respectively. We define the {\it knot-theoretic ternary quasigroup homology} as this last homology, $H^N(X)$.
\end{definition}

\subsection{Homology of IKTQs}
Recall from the first section, that an involutory knot-theoretic ternary quasigroup (IKTQ) is a ternary quasigroup $(X,T,\eL,T,\R)$ satisfying A3L and A3R. The condition $T=\M$
, equivalent to $a(axbT)bT=x$, for any $a$, $b$, $x\in X$, allows us to define another subcomplex.

\begin{definition}\label{degsI}
For an IKTQ $(X,T,\eL,T,\R)$, 
let $C_n^{I}(X)$ be the subgroup of $C_n(X)$ generated by the sums of the form
$x+x[x_j\mapsto x_{j-1}x_jx_{j+1}T]$, i.e.,
\[
(x_0,\ldots,x_{j-1},x_j,x_{j+1},\ldots,x_n)+
(x_0,\ldots,x_{j-1},x_{j-1}x_jx_{j+1}T,x_{j+1},\ldots,x_{n+1}),
\] 
for some $j\in\{1,\ldots,n\}$ and $n\geq 1$. For $n<1$, we take $C_n^{I}(X)=0$ . 
\end{definition}

\begin{theorem}\label{degIL}
For an IKTQ,
\[\partial_n^L(C_n^{I})\subset C_{n-1}^{I}.\]
\end{theorem}
\begin{proof}
Let $x[j]$ denote $(x_0,\ldots,x_{j-1},x_{j-1}x_jx_{j+1}T,x_{j+1},\ldots,x_{n+1})$.
Again we use the fact that $d_i^{n,L}(x_0,\ldots,x_{n+1})$ ends with the sequence
$x_{i+1},\ldots,x_{n+1}$. It follows that $d_i^{n,L}(x+x[j])$, for $i\in\{0,\ldots, j-2\}$, is in the required form.
Now let $i=j$:
\begin{align*}
&d_{j}^{n,L}x=d_{j}^{n,L}(x_0,\ldots,x_{j-1},x_j,x_{j+1},\ldots,x_{n+1})\\
& =d_{j-1}^{n,L}(x_0,\ldots,x_{j-1},x_{j-1}x_jx_{j+1}T,x_{j+1},\ldots,x_{n+1})\\
& =d_{j-1}^{n,L}(x[j]),\\
&d_{j}^{n,L}(x[j])=d_{j}^{n,L}(x_0,\ldots,x_{j-1},x_{j-1}x_jx_{j+1}T,x_{j+1},\ldots,x_{n+1})\\
& =d_{j-1}^{n,L}(x_0,\ldots,x_{j-1},x_{j-1}(x_{j-1}x_jx_{j+1}T)x_{j+1}T,x_{j+1},\ldots,x_{n+1})\\
& =d_{j-1}^{n,L}(x_0,\ldots,x_{j-1},x_j,x_{j+1},\ldots,x_{n+1})=
d_{j-1}^{n,L}x.
\end{align*}
Thus,
\[
d_{j}^{n,L}(x+x[j])=d_{j-1}^{n,L}(x+x[j]),
\]
but in $\partial_n^L$ they appear with opposite signs. Now let $j+1\leq i \leq n$. We will show that
\[
d_i^{n,L}(x+x[j])=d_i^{n,L}x+(d_i^{n,L}x)[j+1],
\]
that is, we need 
\begin{align*}
& d_i^{n,L}(x[j])=(d_{i,1}^{n,L}(x[j]),\ldots,d_{i,j}^{n,L}(x[j]),d_{i,j+1}^{n,L}(x[j]),d_{i,j+2}^{n,L}(x[j]),\ldots,d_{i,n+1}^{n,L}(x[j]))\\
& =(d_{i,1}^{n,L}x,\ldots,d_{i,j}^{n,L}x,(d_{i,j}^{n,L}x)(d_{i,j+1}^{n,L}x)(d_{i,j+2}^{n,L}x)T,d_{i,j+2}^{n,L}x,\ldots,d_{i,n+1}^{n,L}x)\\
& =(d_i^{n,L}x)[j+1].
\end{align*}
The main equalities between coordinates that we will show are
\[
d_{i,j+1}^{n,L}(x[j])=(d_{i,j}^{n,L}x)(d_{i,j+1}^{n,L}x)(d_{i,j+2}^{n,L}x)T
\]
and
\[
d_{i,j}^{n,L}(x[j])=d_{i,j}^{n,L}x.
\]
Since $x$ and $x[j]$ differ only on the $j$-th coordinate, from the inductive definition
\begin{equation*}
d_{i,k}^{n,L}x = \left\{
\begin{array}{rl}
x_{k-1}x_k(d_{i,k+1}^{n,L}x)T & \text{if } k\leq  i\\
x_k & \text{if } k > i
\end{array} \right.
\end{equation*}
follows the equality of the corresponding coordinates $j+2,\ldots,n+1$ in 
$d_i^{n,L}(x[j])$ and $(d_i^{n,L}x)[j+1]$.
For example:
\[
d_{i,j+2}^{n,L}(x[j])=(x[j])_{j+1}(x[j])_{j+2}(d_{i,j+3}^{n,L}(x[j]))T=
x_{j+1}x_{j+2}(d_{i,j+3}^{n,L}x)T=d_{i,j+2}^{n,L}x.
\]
Here $d_{i,j+3}^{n,L}(x[j])=d_{i,j+3}^{n,L}x$ is by induction from higher coordinates, since
\[
d_{i,n+1}^{n,L}(x[j])=d_{i,n+1}^{n,L}x=x_{n+1}.
\]
In the following calculations, we use the axioms A3L and A3R.
\begin{align*}
& d_{i,j+1}^{n,L}(x[j])=(x[j])_j(x[j])_{j+1}(d_{i,j+2}^{n,L}(x[j]))T\\
& =(x_{j-1}x_jx_{j+1}T)x_{j+1}(d_{i,j+2}^{n,L}x)T\\
& =\{x_{j-1}x_j[x_jx_{j+1}(d_{i,j+2}^{n,L}x)T]T\}[x_jx_{j+1}(d_{i,j+2}^{n,L}x)T](d_{i,j+2}^{n,L}x)T\\
& =(d_{i,j}^{n,L}x)(d_{i,j+1}^{n,L}x)(d_{i,j+2}^{n,L}x)T,\\
& d_{i,j}^{n,L}(x[j])=(x[j])_{j-1}(x[j])_{j}(d_{i,j+1}^{n,L}(x[j]))T\\
& =(x[j])_{j-1}(x[j])_{j}[(x[j])_{j}(x[j])_{j+1}(d_{i,j+2}^{n,L}(x[j]))T]T\\
& =x_{j-1}(x_{j-1}x_{j}x_{j+1}T)[(x_{j-1}x_{j}x_{j+1}T)x_{j+1}(d_{i,j+2}^{n,L}(x[j]))T]T\\
& =x_{j-1}x_j[x_{j}x_{j+1}(d_{i,j+2}^{n,L}x)T]T=
x_{j-1}x_j(d_{i,j+1}^{n,L}x)T=d_{i,j}^{n,L}x.
\end{align*}
The equality of the corresponding coordinates $1,\ldots,j-1$ in 
$d_i^{n,L}(x[j])$ and $(d_i^{n,L}x)[j+1]$ now follows from the already shown equalities, for example:
\[
d_{i,j-1}^{n,L}(x[j])=(x[j])_{j-2}(x[j])_{j-1}(d_{i,j}^{n,L}(x[j]))T
=x_{j-2}x_{j-1}(d_{i,j}^{n,L}x)T=d_{i,j-1}^{n,L}x.
\]
\end{proof}
\begin{theorem}\label{degIR}
For an IKTQ,
\[\partial_n^R(C_n^{I})\subset C_{n-1}^{I}.\]
\end{theorem}
\begin{proof}
We will add the symbol of the used operation where needed.
Let $x+x[j]\in C_n^{I,T}(X).$ Then
\[(x+x[j])^r=x^r+(x^r)[n-j]_{\hat{T}},\]
that is, the reversing operator $r$ sends a chain degenerate with respect to $T$ to a chain degenerate with respect to $\hat{T}$, and vice versa. Now the equality
\[ \partial_n^{R,T}x=(-1)^n(\partial_n^{L,\hat{T}}x^r)^r,\]
together with Lemma \ref{rev1} and Theorem \ref{degIL}, finishes the proof.
\end{proof}

From Theorems \ref{degIL} and \ref{degIR} follows 
\begin{theorem}\label{degILR}
For an IKTQ,
\[\partial_n(C_n^{I})\subset C_{n-1}^{I}.\]
\end{theorem}

\begin{definition}\label{degsID}
For an IKTQ $(X,T,\eL,T,\R)$, and $n\geq 1$,
let $C_n^{ID}(X)$ be the subgroup of $C_n(X)$ generated by the sums of the form
$x+x[j]$
for some $j\in\{1,\ldots,n\}$ and an $(n+2)$-tuple $x\in C_n(X)$, and by the $(n+2)$-tuples $y\in C_n(X)$ satisfying at least one of the conditions $D1$ and $D2$ from Definition \ref{degs}.
For $n<1$, we set $C_n^{ID}(X)=0$.
\end{definition}

From Theorems \ref{degLR} and \ref{degILR} follows 
\begin{theorem}\label{degIDLR}
For an IKTQ,
\[\partial_n(C_n^{ID})\subset C_{n-1}^{ID}.\]
\end{theorem}

\begin{definition}
We proved that, for an IKTQ $X$, $(C_n^I(X),\partial_n^L)$, $(C_n^I(X),\partial_n^R)$,
$(C_n^I(X),\partial_n)$,
$(C_n^{ID}(X),\partial_n^L)$, $(C_n^{ID}(X),\partial_n^R)$, and 
$(C_n^{ID}(X),\partial_n)$
are chain subcomplexes of 
$(C_n(X),\partial_n^L)$, $(C_n(X),\partial_n^R)$,
$(C_n(X),\partial_n)$,
$(C_n(X),\partial_n^L)$, $(C_n(X),\partial_n^R)$, and 
$(C_n(X),\partial_n)$, respectively. We denote their homology by $H^{LI}(X)$, $H^{RI}(X)$, $H^I(X)$, $H^{LID}(X)$, $H^{RID}(X)$, and $H^{ID}(X)$,
respectively. We define quotient complexes
\begin{align*} 
& (C_n^{NI}(X),\partial_n^L)=(C_n(X)/C_n^I(X),\partial_n^L),\\
& (C_n^{NI}(X),\partial_n^R)=(C_n(X)/C_n^I(X),\partial_n^R),\\ 
& (C_n^{NI}(X),\partial_n)=(C_n(X)/C_n^I(X),\partial_n),\\
& (C_n^{NID}(X),\partial_n^L)=(C_n(X)/C_n^{ID}(X),\partial_n^L),\\
& (C_n^{NID}(X),\partial_n^R)=(C_n(X)/C_n^{ID}(X),\partial_n^R),\\ 
& (C_n^{NID}(X),\partial_n)=(C_n(X)/C_n^{ID}(X),\partial_n),
\end{align*}
with induced differentials (and the same notation). We denote the homology of these complexes by $H^{LNI}(X)$, $H^{RNI}(X)$, $H^{NI}(X)$,
$H^{LNID}(X)$, $H^{RNID}(X)$, and $H^{NID}(X)$, respectively.
We define homology with coefficients other than $\Z$, and cohomology, in a standard way.
\end{definition}

\section{Geometric interpretation}

Let $D$ be a link diagram (or a flat link diagram) on a compact oriented surface $F$, or on a plane, colored by elements of KTQ (or IKTQ). Then we can assign to it a cycle with respect to the differential $\partial$ in one of the homology theories that we defined in this paper. The kind of homology that we choose depends on what kind of invariants (that is, invariants under which moves) we want to obtain. 

\begin{figure}
\begin{center}
\includegraphics[height=3.5 cm]{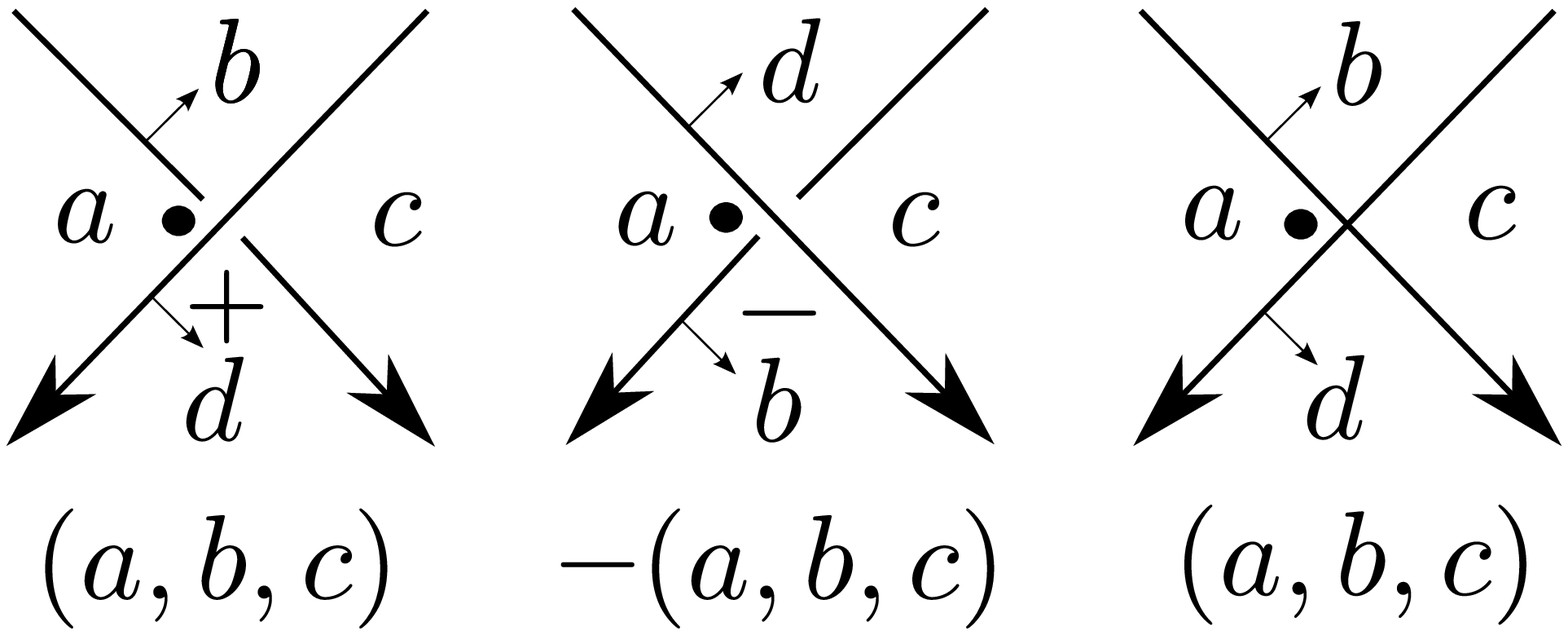}
\caption{}\label{chains}
\end{center}
\end{figure}

\begin{definition}
Figure \ref{chains} shows the way of assigning a signed chain (a signed triple) to a colored classical positive crossing, a classical negative crossing, and a flat crossing. The chain assigned to the entire diagram is the sum of such signed expressions taken over all classical (or flat) crossings. We will call such a chain for a diagram $D$ on a compact oriented surface $F$ (or a plane)  an {\it associated chain} $c_D$.
\end{definition}

\begin{figure}
\begin{center}
\includegraphics[height=6.5 cm]{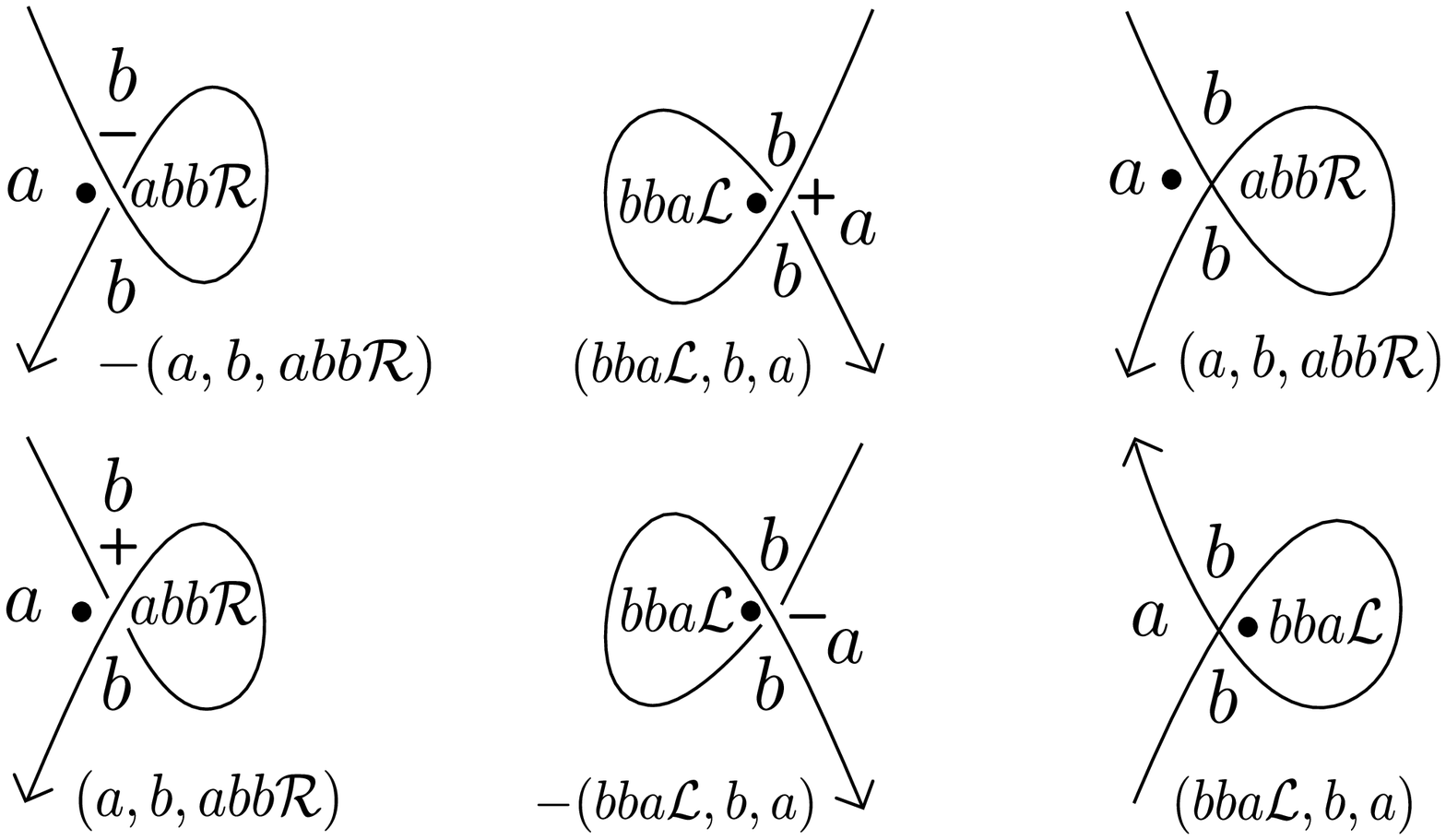}
\caption{}\label{degR1}
\end{center}
\end{figure}

\begin{figure}
\begin{center}
\includegraphics[height=6 cm]{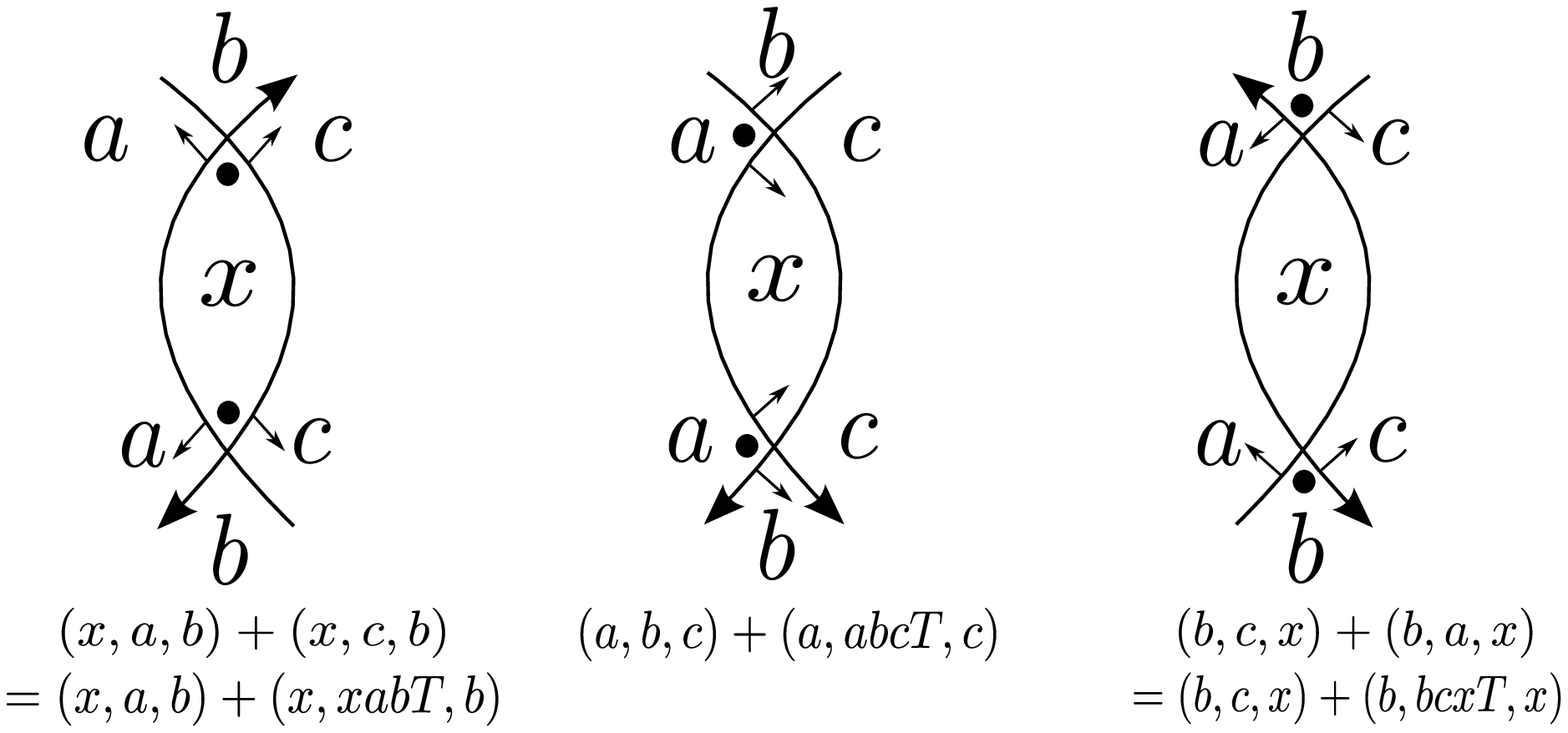}
\caption{}\label{fR2deg}
\end{center}
\end{figure}

From the Figures \ref{degR1} and \ref{fR2deg}, we see that the subcomplex 
$C_n^D$ corresponds to the first Reidemeister move (and its flat version),
the subcomplex $C_n^I$ corresponds to the second flat Reidemeister move, and 
the subcomplex $C_n^{ID}$ is related to the first and second flat Reidemeister moves. 

\begin{figure}
\begin{center}
\includegraphics[height=4.5 cm]{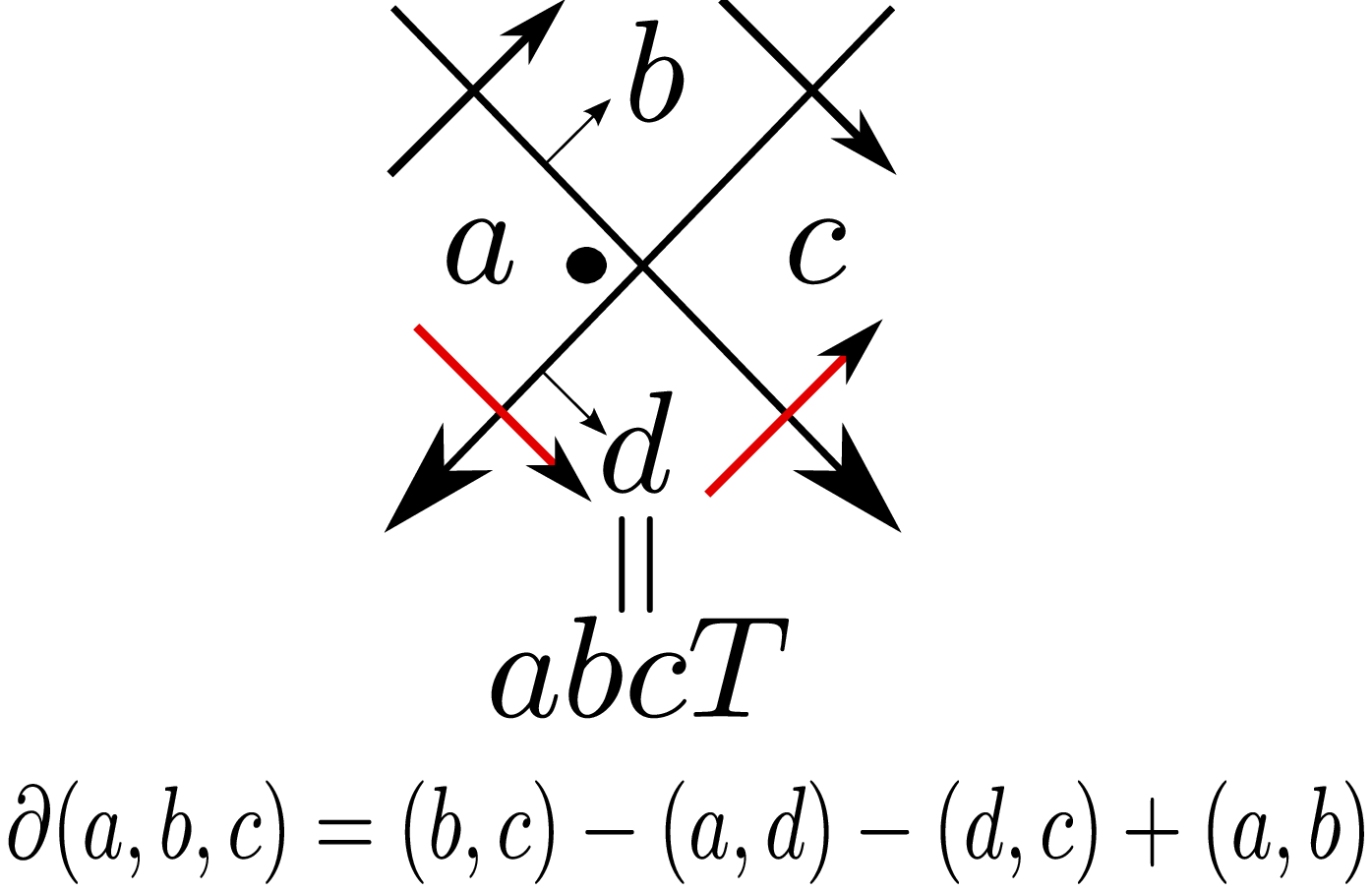}
\caption{}\label{flatdiff}
\end{center}
\end{figure}

\begin{figure}
\begin{center}
\includegraphics[height=3.5 cm]{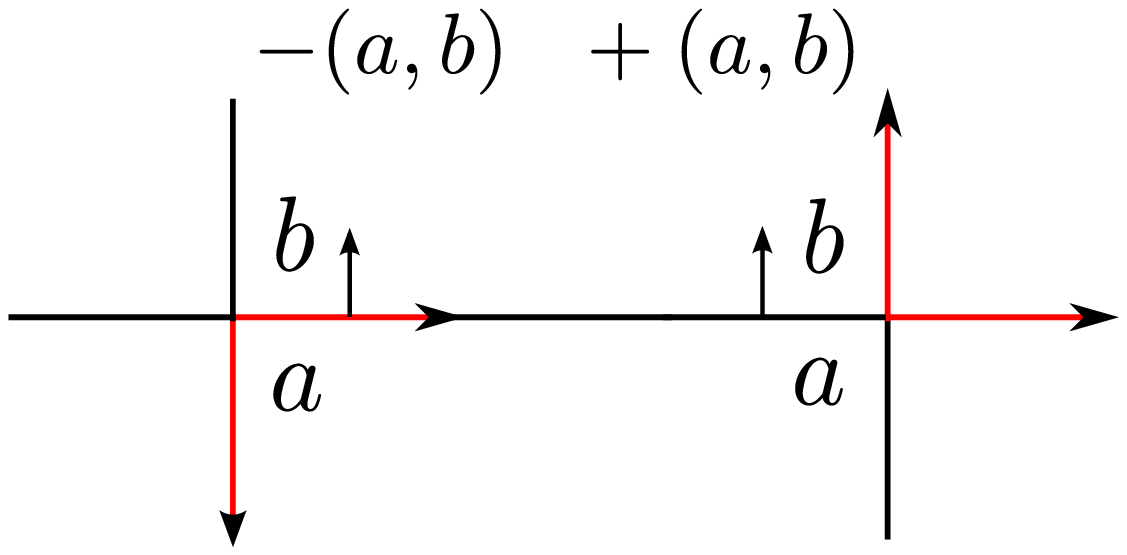}
\caption{}\label{fcycle}
\end{center}
\end{figure}

\begin{lemma} For an IKTQ-colored flat oriented link diagram $D$ on a compact oriented surface $F$, its associated chain $c_D$ is a cycle in IKTQ homology of the IKTQ used to color the diagram.
\end{lemma} 
\begin{proof}
At every flat crossing, there are two incoming and two outgoing edges. In the differential of a triple of colors assigned to a flat crossing, the pairs of colors surrounding outgoing edges get a negative sign, and the other ones are positive, see Fig. \ref{flatdiff}.
Thus, an edge which is outgoing for one flat crossing, at its next flat crossing is incoming, and yields a positive pair of the same colors, see Fig. \ref{fcycle}. Therefore, $\partial(c_D)=0$, i.e., the entire colored diagram represents a cycle. 
\end{proof}

\begin{figure}
\begin{center}
\includegraphics[height=6 cm]{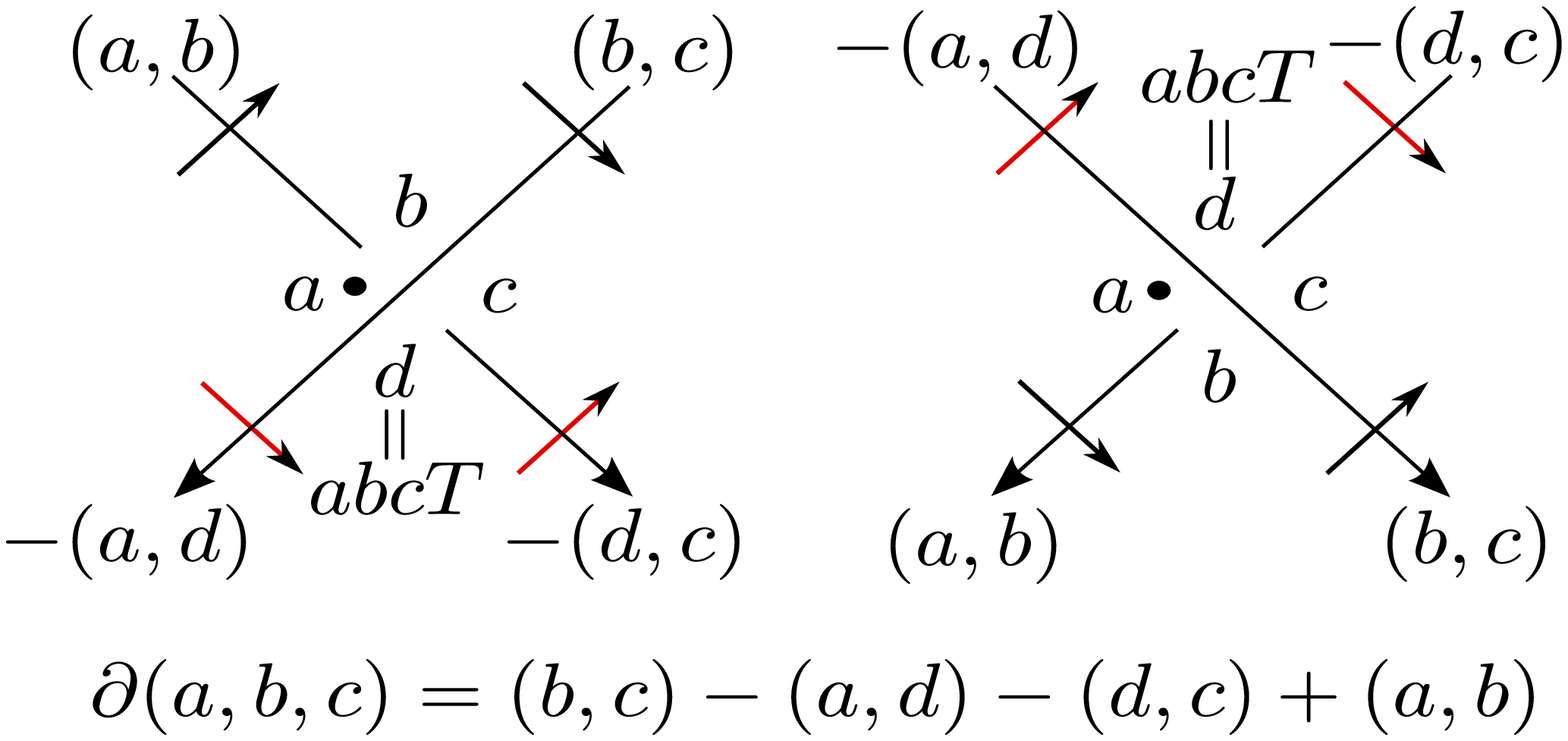}
\caption{}\label{quasidiff}
\end{center}
\end{figure}

\begin{lemma} For an KTQ-colored oriented link diagram $D$ on a compact oriented surface $F$ (or a plane), its associated chain $c_D$ is a cycle in KTQ homology of the KTQ used to color the diagram.
\end{lemma} 
\begin{proof}
We see from the Figure \ref{quasidiff}, that in the differential 
\begin{align*}
\partial(a,b,c)&= (b,c)-(a,abcT)-(abcT,c)+(a,b)\\
&=(b,c)-(a,d)-(d,c)+(a,b)
\end{align*}
of a triple of colors assigned to a positive crossing, the positive pairs of colors correspond to the incoming edges, and the negative pairs can be assigned to the outgoing edges. For a negative crossing, the chain assigned to it is $-(a,b,c)$. Thus,
in $\partial[-(a,b,c)]=-(b,c)+(a,d)+(d,c)-(a,b)$, again the positive pairs correspond to the incoming edges, and the negative pairs to the outgoing edges. Also, the order of colors in each pair is pointed by the co-orientation of a diagram.
This ensures that the two ends of an edge give the same pair of colors, but with opposite signs. Thus, $\partial(c_D)=0$. 
\end{proof}

\begin{lemma}\label{flathominv}
The IKTQ homology class of a cycle assigned to an oriented IKTQ colored flat link diagram on a compact oriented surface $F$ is an invariant under all flat Reidemeister moves if we use the homology $H^{NID}(X)$, and under the second and third flat Reidemeister moves, if we use $H^{NI}(X)$.
\end{lemma}

\begin{proof}
In case of homology $H^{NID}(X)$, the first flat Reidemeister move adds or removes a degenerate cycle of the form ($a$, $b$, $abb\R$) or ($bba\eL$, $b$, $a$), see Fig. \ref{degR1}, so it does not change the homology class.

For $H^{NID}(X)$ and $H^{NI}(X)$, the second flat Reidemeister move adds or removes a degenerate cycle of the form $(x,y,z)+(x,xyzT,z)$, see Fig. \ref{fR2deg}, so the homology class is not changed.

Now consider the third flat Reidemeister move (Fig. \ref{fR3}).
For $H^{NID}(X)$ and $H^{NI}(X)$, the contributions from the crossings after the move, minus the contributions before the move, are equal to the boundary:
\begin{align*}
\partial_2(a,b,c,d) & =(b,c,d)-(a,abcT,(abcT)cdT)\\
& -(abcT,c,d)+(a,b,bcdT)\\
& +(ab(bcdT)T,bcdT,d)-(a,b,c).
\end{align*}
We also used the fact, that the four oriented Reidemeister moves of type 2, together with any one of the eight oriented Reidemeister moves of type 3, are sufficient to generate the other seven Reidemeister moves of type 3. 
\end{proof}

In a similar way we can prove the following lemma.

\begin{lemma}\label{vhominv}
The KTQ homology class of a cycle assigned to an oriented KTQ-colored link diagram on a compact oriented surface $F$ (or a plane) is an invariant under all Reidemeister moves if we use the homology $H^{N}(X)$. If we do not need the invariance under the first Reidemeister move, then $H(X)$ can be used (i.e., the basic homology corresponding to the differential $\partial$, without the use of any quotient complexes).
\end{lemma}
\begin{proof}
The main difference between the proof of this statement, and the proof in Lemma \ref{flathominv} concerns the second Reidemeister move. This time, we can use the signs of classical crossings. The contributions coming from the two crossings in the second Reidemeister move cancel out, because the crossings have opposite signs. 
\end{proof}

The applications of these homologies to Yoshikawa diagrams, and to broken surface diagrams, will be described in another paper.

In an analogous way to the construction in \cite{CJKLS03}, we define cocycle invariants.

\begin{definition}
Let $\phi$ be a cocycle from a KTQ (or IKTQ) cohomology of $X$, with inputs from $C_1(X)$, and taking values in some abelian group $A$ written multiplicatively.
For a coloring $\mathcal{C}$ of a diagram $D$ on a compact oriented surface $F$, or a plane, using $X$,
a {\it Boltzmann weight}, $B(\tau,\mathcal{C})$, assigned to a classical or flat crossing $\tau$, is the value of the cocycle on the signed triple of colors associated with a crossing as in Fig. \ref{chains}.
That is, we take $B(\tau,\mathcal{C})=\phi(a,b,c)$ or $B(\tau,\mathcal{C})=\phi(a,b,c)^{-1}$, depending on the type of crossing. The {\it cocycle knot invariant} is defined by the state-sum expression
\[ 
\Phi(D)=\sum_{\mathcal{C}}\prod_{\tau}B(\tau,\mathcal{C}),
\]
where the product is taken over all classical or flat crossings of $D$, and the sum is taken over all colorings of $D$ with $X$. The value of $\Phi(D)$ is in the group ring $\Z[A]$.
\end{definition}
Depending on which cohomology we use, $\Phi(D)$ will be invariant under the corresponding family of Reidemeister moves, or flat Reidemeister moves, as follows from the constructions of our homologies, and Lemmas \ref{flathominv} and \ref{vhominv}.

\bibliography{simplex}
\bibliographystyle{plain}
\end{document}